\documentclass[reqno,11pt]{amsart}

\numberwithin{equation}{section}

\usepackage[backend=bibtex,style=alphabetic,natbib=true,url=false]{biblatex}
\usepackage[utf8]{inputenc}
\usepackage[italian,english]{babel}
\usepackage{amsmath}
\usepackage{amsfonts}
\usepackage{amssymb,esint}
\usepackage{tikz}
\usepackage{geometry}
\usepackage{color}
\usepackage{mathrsfs}
\usepackage{mathtools}

\mathtoolsset{showonlyrefs}

\usepackage[colorlinks, citecolor=citegreen, linkcolor=refred, urlcolor=blue]{hyperref}
\definecolor{citegreen}{rgb}{0,0.4,0}
\definecolor{refred}{rgb}{0.5,0,0}

\theoremstyle{plain}
\newtheorem {theorem}{Theorem}[section]

\newtheorem {proposition} [theorem]{Proposition}

\newtheorem{definition}[theorem]{Definition}
\theoremstyle{remark}
\newtheorem{remark}[theorem]{Remark}

\DeclarePairedDelimiter\abs{\lvert}{\rvert}

\renewcommand{\theequation}{\arabic{section}.\arabic{equation}}

\newcommand{\R}{\mathbb R}
\newcommand{\N}{\mathbb N}

\renewcommand{\theta}{\vartheta}

\newcommand{\barint}
{\rule[.036in]{.12in}{.009in}\kern-.16in \displaystyle\int}

\newcommand{\dive}{{\mathrm{div}}}

\usepackage{color}

\newcommand{\numberset}{\mathbb}
\renewcommand{\N}{\numberset{N}}
\renewcommand{\R}{\numberset{R}}
\newcommand{\Sf}{\numberset{S}}

\newcommand{\Heis}{\numberset{H}}

\newcommand{\ric}{\mathop {\rm Ric}\nolimits}

\newcommand{\dd}{{\,\rm d}}
\newcommand{\HH}{{\rm H}}
\newcommand{\hh}{{\rm h}}

\renewcommand{\phi}{\varphi}
\renewcommand{\epsilon}{\varepsilon}

\bibliography{example}

\title[Substatic Riemannian manifolds and the Alexandrov Theorem]{New integral estimates in substatic Riemannian manifolds and the Alexandrov Theorem}

\author[M.~Fogagnolo]{Mattia Fogagnolo}
\address{M.~Fogagnolo, Centro di Ricerca Matematica Ennio De Giorgi, Scuola Normale Superiore, Piazza dei Cavalieri, 3, 56126 Pisa, Italy}
\email{mattia.fogagnolo@sns.it}

\author[A.~Pinamonti]{Andrea Pinamonti}
\address{A.~Pinamonti, Universit\`a degli Studi di Trento,
via Sommarive 14, 38123 Povo (TN), Italy}
\email{andrea.pinamonti@unitn.it}
\thanks{A. Pinamonti is partially supported
by the INdAM-GNAMPA 2020 project 
\emph{Convergenze variazionali per funzionali e operatori dipendenti da campi vettoriali.}}
\begin{document}

\begin{abstract}
We derive new integral estimates on substatic manifolds with boundary of horizon type, naturally arising in General Relativity. In particular, we generalize to this setting an identity due to Magnanini-Poggesi \cite{magnanini-poggesi1} leading to the Alexandrov Theorem in $\R^n$ and improve on a Heintze-Karcher type inequality due to Li-Xia \cite{Xia-substatic}. Our method relies on the introduction of a new vector field with nonnegative divergence, generalizing to this setting the P-function technique of Weinberger \cite{Wei_ARMA}.
\end{abstract}

\maketitle

\section{Introduction}
The characterization of hypersurfaces with constant mean curvature (and also of other functions of the principal curvatures) is an issue that, since the seminal papers of Alexandrov \cite{alexandrov-originale, alexandrov-originale2} generated a vast and fertile literature on related problems, both because of its intrinsic interest and also because it naturally links different topics, such as boundary value problems and Riemannian geometry. 

We recall some of the progresses sprung from this problem, in order to help framing our work. In the just mentioned initial papers of Alexandrov, the author showed, by means of a brilliant technique now known as \emph{Alexandrov moving planes}, that embedded closed hypersurfaces in $\R^n$ and in space forms (the hyperbolic space and the hemisphere) were shown to be only geodesic balls. Reilly, on the other hand,  proposed an alternative proof \cite{reilly-originale} based on an integral identity now known as Reilly's formula.
After that, inspired by some of the computations worked out by Weinberger \cite{Wei_ARMA} in order to yield an alternative, way easier proof to a special case of Serrin's Theorem about overdetermined problems \cite{Serrin_ARMA}, Reilly gave a third proof \cite{reilly-pfunz}. This approach was indeed ruled by the sub-harmonicity of Weinberger's function 
\[
P = \abs{\nabla u}^2 + \frac{2}{n} u,
\]
where $u$ is the solution to $\Delta u = -1$ in the bounded set $\Omega$ contoured by the hypersurface $\Sigma$ with constant mean-curvature with vanishing boundary value. Carefully retracing this path,  Magnanini and Poggesi \cite[Theorem 2.2]{magnanini-poggesi1} discovered the following remarkable integral identity, holding true for any bounded set $\Omega \subset \R^n$ with smooth boundary $\Sigma$ and mean curvature $\HH$
\begin{equation}
\label{magnanini-poggesif}
       \int_\Sigma  \abs{\nabla u}^2 (\overline{\HH} - \HH) \dd\sigma,
      = \int_\Omega  \left\vert \nabla \nabla u - \frac{\Delta u}{n} g_{\R^n}  \right\vert^2  \dd\mu  \, + \,
       \frac{(n-1)}{n} \frac{1}{R} \int_\Sigma (R - \abs{\nabla u})^2 \dd\sigma, 
 \end{equation}
where $\overline{\HH}$ and $R$ are the values achieved by the mean curvature of $\Sigma$ and $\abs{\nabla u}$ on $\Sigma$, respectively, on a ball of the same volume as $\Omega$. The Riemannian metric $g_{\R^n}$ is the classical flat metric of $\R^n$, while $\dd\sigma$ and $\dd\mu$ denote respectively the $n-1$-dimensional and the $n$-dimensional Hausdorff measures. The above equation  carries a quantitative information about hypersurfaces with almost constant mean curvature, since it shows that the (integral) deficit from being constantly mean-curved controls the $L^2$-norm of the traceless part of the Hessian of $u$, that actually vanishes if and only if $u$ is rotationally symmetric. In particular, \eqref{magnanini-poggesif} represents a generalization of the Alexandrov Theorem holding true for any bounded set with smooth boundary.

Again as a consequence of the sub-harmonicity of $P$, Magnanini and Poggesi got a version of an inequality now known as Heintze-Karcher inequality, originated in \cite{heintze-karcher} and rediscovered in the form that is usually meant in \cite[Theorem 1]{ros}, reading
\begin{equation}
    \label{magna-pg-heintze}
\frac{n-1}{n} \int_{\Sigma}\frac{1}{\HH} \dd\sigma \, \geq \, \abs{\Omega}
\end{equation}
for open subsets $\Omega$ with strictly \emph{mean-convex} (i.e. with strictly positive mean curvature) boundary $\Sigma$.
In fact, a $L^2$-quantity like that in \eqref{magnanini-poggesif} is found to appear as a deficit also the above inequality.

\medskip

\emph{The aim of the present paper is to extend this circle of ideas to the vast context of substatic manifolds with horizon boundary.} 

\smallskip

Leaving the rigorous (and actually more general) definitions to the next section, we consider Riemannian manifolds $(M, g)$ of dimension $n \geq 2$ endowed with a nonnegative smooth function $f$, with a compact, \emph{minimal} (i.e. with vanishing mean curvature) connected boundary $N = \{f = 0 \}$ that is a regular level set of $f$ (i.e. $\abs{\nabla f} > 0$ on $N$). This is substantially what we mean by \emph{horizon boundary}. Moreover, the curvature of $g$ satisfies the relation 
\begin{equation}
    \label{substatic-intro}
    f \ric  -  \nabla \nabla f + (\Delta f) g \geq 0,
\end{equation}
that is what \emph{substatic} stands for. We will refer to $f$ as \emph{substatic potential}. In order to discuss the relevance of this conditions, let us start from the \emph{static} case, that occurs when \eqref{substatic-intro} holds with equality at any point. Static Riemannian metrics naturally arise when considering solutions to the Einstein equations in the vacuum, that is, with vanishing stress-energy tensor. When the resulting Lorentzian manifold is also assumed to admit a global timelike Killing vector field with integrable orthogonal distribution, then time-slices are immediately seen to be static, with actually an additional partial differential constraint on the static potential that in turn implies the boundary to be of horizon-type. Discussing in details the deep connections with Mathematical General Relativity is far out the scope of this contribution, and so we refer the interested reader to the comprehensive thesis of Borghini \cite{stefano} and to the references therein. The substatic condition appear more generally in General Relativity in relation with the so called \emph{null convergence condition} \cite{wang-alexandrov}.

It was also observed by  Brendle \cite{Brendle-alexandrov} that this condition is naturally satisfied on warped product manifolds substantially fulfilling certain natural assumptions on the warping functions, that are clearly satisfied by the rotationally symmetric models of static metrics (de Sitter-Schwarzschild and anti de Sitter-Schwarzschild) as well as by the Reissner-Nordstr\"om metric. In particular, in \cite{Brendle-alexandrov} the author actually provides a version of the Alexandrov Theorem and of the Heintze-Karcher inequality for such warped product metrics. A generalization of the latter in the substatic setting has been obtained in \cite{Xia-substatic}, through the application of a suitable Reilly-type identity. In concluding their proof of the Heintze-Karcher inequality, Li-Xia introduced and considered the fundamental boundary value problem 
\begin{equation}
\label{torsion-pb-intro}    
\begin{cases}
\Delta u = - 1 + \frac{\Delta f}{f} u & \mbox{in} \,\, \Omega 
\\
\,\,\,\,\,u=c & \mbox{on}\,\, N
\\
\,\,\,\,\,u = 0 &\mbox{on} \,\, \Sigma,
\end{cases}
\end{equation}
for a suitable nonnegative constant $c$ on a bounded set $\Omega \subset M$ with smooth boundary given by $\partial \Omega = N \sqcup \Sigma$. This equation constitutes the core of our generalization of \eqref{magnanini-poggesif}. Namely, in our main result Theorem \ref{alexandrov-th} we prove that, under the additional assumption to be discussed below of existence of a strictly mean-convex hypersurface $S$ bounding a subset $\Omega$ with boundary $\partial \Omega = N \sqcup S$, we can choose 
\begin{equation}
    \label{def-cj-intro}
    c = \frac{\int_{N} \abs{\nabla f} \dd\sigma}{\int_{N}\abs{\nabla f}\left[\frac{\Delta f}{f} - \frac{\nabla \nabla f}{f} \left(\frac{\nabla f}{\abs{\nabla f}}, \frac{\nabla f}{\abs{\nabla f}}\right)\right] \dd\sigma}
    \end{equation}
    and that with this choice we have
\begin{equation}
    \label{main-id-intro}
    \begin{split}
     0 &\leq \int_\Omega  f \left\vert \nabla \nabla u - \frac{\Delta u}{n} g - u\left(\frac{\nabla \nabla f}{f} - \frac{\Delta f}{nf} g\right) \right\vert^2  + Q\left(\nabla u - \frac{u}{f} \nabla f, \nabla u - \frac{u}{f} \nabla f\right) \dd\mu \\ 
     & \qquad\qquad\qquad\qquad +  \frac{(n-1)}{n} \frac{1}{R} \int_\Sigma f(R - \abs{\nabla u})^2 \dd\sigma 
     =\int_\Sigma f \abs{\nabla u}^2 (\overline{\HH} - \HH) \dd\sigma,
    \end{split}
    \end{equation}
    where
   $Q = f\ric - \nabla\nabla f +\Delta f g \geq 0$
    and
    \begin{equation}
        \label{ReH-intro}
        R= \frac{\int_\Omega f\dd\mu + c\int_{N} \abs{\nabla f}\dd\sigma}{\int_{\Sigma} f \dd\sigma} \qquad \overline{\HH} = \frac{(n-1)}{n} \frac{1}{R}.
    \end{equation}
It is then easy to show that if $\HH \geq \overline{\HH}$ on the whole hypersurface $\Sigma$ then it is totally umbilical. To our knowledge, the above identity is new also in the more symmetric case of the warped product metrics considered in \cite{Brendle-alexandrov}. We included an Appendix where we show how Brendle's Alexandrov Theorem  
can be derived from the results contained in our paper. In fact, our notion of substatic metrics with horizon boundary (Definitions \ref{substatic-def} and \ref{horizon}) holds true also for empty horizon boundaries, in order to encompass also the hemisphere and the hyperbolic space as well as the huge and thoroughly studied class of complete manifolds with nonnegative Ricci curvature.

As for \eqref{magnanini-poggesif}, the relation \eqref{main-id-intro} naturally yields a natural deficit for $\Sigma$ from being, in loose integral sense, of constant mean curvature $\overline{\HH}$. Namely, by applying the H\"older inequality to \eqref{main-id-intro}, we get
\begin{equation}
    \label{quantitative}
    \big\vert\big\vert \HH - \overline{\HH} \big\vert\big\vert_{L^{p}(\Sigma)}  \geq \frac{\int_\Omega  f \left\vert \nabla \nabla u - \frac{\Delta u}{n} g - u\left(\frac{\nabla \nabla f}{f} - \frac{\Delta f}{nf} g\right) \right\vert^2  + Q\left(\nabla u - \frac{u}{f} \nabla f, \nabla u - \frac{u}{f} \nabla f\right) \dd\mu}{\vert\vert f \abs{\nabla u}^2 \vert\vert_{L^{p'}(\Sigma)}},
\end{equation}
where $p \geq 1$ and $p'$ is its H\"older conjugate.
In flat $\R^n$, where the numerator on the right hand side reduces to the $L^2$-norm of the traceless hessian of $u$, inequalities of the above form have been exploited in order to get a more quantitative description in terms of a suitable distance between $\Sigma$ and a suitable reference ball, see  \cite{magnanini-poggesi1,magnanini-poggesi2, magnanini-poggesi3,ciraolo-maggi,vesa}. In the hyperbolic space and in the hemisphere, this approach has been generalized in the recent \cite[Theorem 1.5]{scheuer-alexandrov}, where the role of the $L^2$-norm of the traceless Hessian of a suitable function of $u$ has been largely discussed and exploited. Providing similar results using the deficit in \eqref{quantitative} looks quite intriguing, as it would for example yield a complete quantitative description of Brendle's Alexandrov Theorem \cite{Brendle-alexandrov} in substatic warped products. Concerning other quantitative versions of the Alexandrov Theorem, this time obtained through the Alexandrov moving planes method, we mention the sharp \cite{ciraolo-vezzoni-alexandrov} in flat $\R^n$ and \cite{ciraolo-vezzoni-hyperbolic,ciraolo-roncoroni-spaceforms} about space forms. 

\medskip

As already mentioned, we prove our result by providing a suitable generalization of the $P$-function computation of Weinberger as a sort of pointwise alternative to Li-Xia's Reilly-type integral identity \cite{Xia-substatic}. What we discover is a new vector field $X$ that in the substatic case has nonnegative divergence, that substantially coincides with the integrand in the numerator on the right hand side of \eqref{quantitative}. Such vector field coincides with $\nabla P$ on manifolds with nonnegative Ricci curvature, and it is nontrivially linked  (Remark \ref{XeP}) with the  P-function of constant curvature spaces \cite{qiu-xia-pfunz,Ciraolo-Vezzoni-Serrin}.
In static and substatic manifolds with harmonic potential, a very different vector field with nonnegative divergence have been found and utilized for different aims in \cite{Ago_Maz_2,agostiniani-mazzieri-oronzio}.
Identity \eqref{main-id-intro} follows very naturally from applying the Divergence Theorem to the vector field $X$ in $\Omega$ (Proposition \ref{integral-id}) and working out some basic algebraic manipulations. A contradiction argument exploiting the existence of a strictly mean-convex hypersurface $S$ yields the positivity of the constant in \eqref{def-cj-intro}, see Proposition \ref{boundary-positive}. 
Observe that this is a purely geometric consequence obtained for the horizon boundary of a substatic manifold, that to our knowledge has only been pointed out without an explicit proof in \cite[Theorem 1.6]{wang-brendle}, in the special static geometry. In such context the result appears more elegant, and it is discussed in Remark \ref{remark-static} 
We also provide a geometric condition ensuring \emph{a priori} the existence of strictly mean-convex hypersurfaces like these, see Proposition \ref{existence-convex}, relying on the resolution of the least area problem recently considered in \cite{Fog_Maz}. Without additional effort, the Heintze-Karcher-type inequality for strictly mean-convex $\Sigma$
\begin{equation}
    \label{heintze-vera-intro}
     \frac{n-1}{n} \int_{\Sigma}\frac{f}{\HH} \dd\sigma \geq\int_{\Omega} f \dd\mu +  c \int_{N} \abs{\nabla f} \dd\sigma
\end{equation}
is derived, again with an explicit deficit, see Theorem \ref{heintze-karcher}. The constant $c$ is again given by \eqref{def-cj-intro}, and it is well defined as $\Sigma$ itself provides the strictly mean-convex hypersurface required. An analogous inequality in substatic manifolds is provided in \cite[Theorem 1.3]{Xia-substatic}, as already said as an application of a Reilly-type identity. However, the constant $c$ considered there is not known in general to be positive, and it can be checked that even when this is the case, it is a priori smaller then ours. On the other hand, they do coincide in the warped product models considered in \cite{Brendle-alexandrov}.

\bigskip

We close this introduction with a brief summary. In Section \ref{sec2}, after having introduced the setting and the basics about problem \eqref{torsion-pb-intro}, we introduce the key vector field $X$. In Section \ref{sec3}, we work out the main results of the paper, already discussed above. Moreover, in Subsection \ref{subsec}, we derive stronger rigidity statements when \eqref{heintze-vera-intro} holds with equality under additional geometric assumptions. In the Appendix, we discuss Brendle's Alexandrov Theorem in connection with the present work.



\bigskip

\bigskip
\emph{After the manuscript has been uploaded on ArXiv, Prof. C. Xia explained us that the fact that the vector field $X$ defined in \eqref{fieldX} has nonnegative divergence could be also deduced from the computations in \cite{Xia-substatic, li-xia-affine}.}  

\bigskip

\textbf{Acknowledgements.} The authors would like to thank R. Magnanini, L. Mazzieri and G. Poggesi for their interest in the present paper. Moreover, the authors are indebted to J. Li and C. Xia for having pointed out the papers \cite{qiu-xia-pfunz, li-xia-affine}  and for the very interesting discussions we had.

\section{A vector field with nonnegative divergence in substatic manifolds}
\label{sec2}
Let us provide the rigorous definitions we are going to employ for substatic manifolds with horizon boundary. For all the duration of the paper, \emph{we are considering Riemannian manifolds of dimension $n \geq 2$}.
\begin{definition}[Horizon boundaries]
\label{horizon}
We say that a  Riemannian manifold $(M, g)$ endowed with a smooth nonnegative function $f \in C^{\infty}(M)$ has \emph{horizon boundary} if $\nabla \nabla f /f$ extends continuously  at points $x \in M$ where $f(x) = 0$ and one of the following alternatives occurs.
\begin{itemize}
\item[(i)] The function $f$ is \emph{proper} and the boundary $(N, g_N)$ of $(M, g)$ is a \emph{minimal,} smooth, closed $(n-1)$-dimensional Riemannian submanifold such that $N = \{f = 0\}$ and it is a regular level set for $f$. In this case, we denote with $N_1, \dots, N_l$, with $l \in \N$, the connected components of $N$.
    \item[(ii)] The boundary of $(M, g)$ is empty and $f$ is strictly positive.   
\end{itemize}
\end{definition}
A couple of comments are in order. There is no completeness request in the definition above. This is not anecdotic at all, since it allows to consider fundamental examples like the De Sitter-Schwarzschild metric, that is a specific warped product of the type $([0, \overline{\rho}), d\rho \otimes d\rho + h^2(\rho) g_N$, where $N$ is a closed $n-1$ dimensional closed hypersurface  that has horizon boundary $\{f = 0\} = \{\rho = 0\} = N$, where $f = h'(\rho)$. This metric is incomplete as $\rho \to \overline{\rho}^-$.

The very same thing happens for the hemisphere, that we consider devoid of the equatorial $n-1$-dimensional hypersphere. It is again a warped product of the form above with a different warping function $h$, with $f = h'(\rho)$. In this case, the metric closes smoothly as $\rho \to 0^+$. The boundary is thus empty and  (ii) in the above definition is satisfied. 

These metrics are also \emph{substatic}. 
\begin{definition}[Substatic metrics]
\label{substatic-def}
Let $(M, g)$ be a Riemannian manifold endowed with a nonnegative function $f \in C^{\infty}(M)$, and with horizon boundary. We say that it is \emph{substatic} if
\begin{equation}
    f \ric - \nabla \nabla f + \Delta f g \geq 0
    \end{equation}
    on the whole of $M$.
    In this case, we say that $f$ is the \emph{substatic potential} of $(M, g)$.
\end{definition}
In addition to the substatic warped products considered in \cite{Brendle-alexandrov} (see also the Appendix), we mention that \emph{complete} noncompact manifolds with nonnegative Ricci curvature and without boundary fulfil the assumptions above by choosing $f$ to be a positive constant.

\subsection{The torsion-like function}
Let $(M, g)$ be a  Riemannian manifold endowed with a nonnegative function $f \in C^{\infty} (M)$ and with horizon boundary $(N, g_N)$. 
Our first concern consists in providing, under such assumption, the basic properties of solutions to 
\begin{equation}
\label{torsion-pb}    
\begin{cases}
\Delta u = - 1 + \frac{\Delta f}{f} u & \mbox{in} \,\, \Omega 
\\
\,\,\,\,\,u=c_j & \mbox{on}\,\, N_j
\\
\,\,\,\,\,u = 0 &\mbox{on} \,\, \Sigma,
\end{cases}
\end{equation}
where $\Omega \subset M$ is a bounded open subset with $C^{2,\alpha}$- boundary such that $\partial \Omega = \Sigma \sqcup (\cup_{j \in J} N_j)$ for $J \subseteq \{1, \dots, n\}$, with  $\Sigma \cap N = \emptyset$ and $c_j > 0$ are positive constants.
In case $\Omega \Subset M \setminus N$ (that is $J = \emptyset$), or  $N$ were empty (that is, (ii) occurs in Definition \ref{horizon}),  then we agree that the problem \eqref{torsion-pb} reduces to 
\begin{equation}
\label{torsion-easy}    
\begin{cases}
\Delta u = - 1 + \frac{\Delta f}{f} u & \mbox{in} \,\, \Omega 
\\
\,\,\,\,\,u = 0 &\mbox{on} \,\, \Sigma.
\end{cases}
\end{equation}
In the following statement we  observe that \eqref{torsion-pb} admits a positive smooth solution in $\Omega$ that is $C^{2, \alpha}$ up to the boundary, and that also enjoys a Hopf-type property on $\Sigma$. The nonstandard maximum principle lying behind the nonnegativity and the Hopf-property of $u$ follows from an argument that is inspired from a similar one in the proof of \cite[Lemma 2.4]{Ciraolo-Vezzoni-Serrin}, in the special geometry of a hemisphere.
\begin{theorem}
\label{existence}
Let $(M, g)$ be a Riemannian manifold endowed with a nonnegative smooth function $f \in C^{\infty}(M)$ with horizon boundary. Let $\Omega \subset M$ be a bounded subset with $C^{2, \alpha}-boundary$ such that $\partial \Omega = \Sigma \sqcup (\cup_{j \in J} N_j)$ with $J \subset \{1, \dots, l\}$, and $\Sigma \cap N = \emptyset$.  Then,
there exists an unique solution $u \in C^{\infty}(\Omega) \cap C^{2,\alpha}(\overline{\Omega})$ to \eqref{torsion-pb}. Moreover, $u > 0$ in $\Omega$ and $\Sigma = \{u = 0\}$ is a regular level set of $u$. In fact,
we have
\begin{equation}
\label{hopf}
\frac{\partial u}{\partial \nu}(x) < 0     
\end{equation}
at any $x \in \Sigma$, and $\nu$ is unit normal to $\Sigma$ in $x$ pointing outside of $\Omega$. In particular, $\nu = - \nabla u / \abs{\nabla u}$.
\end{theorem}
\begin{proof}
The existence of a unique classical solution $u$ to \eqref{torsion-pb} with the claimed regularity follows from \cite[Theorem 6.15]{Gil_Tru_book}, once checked that the operator $\Delta - \Delta f /f$ has strictly positive first eigenvalue. This is shown in \cite[Lemma 2.5]{Xia-substatic}.

\smallskip

Assume that $N= \{f = 0\}$ is a smooth hypersurface and regular level set of $f$ and $\partial\Omega = \Sigma \sqcup (\cup_{j \in J} N_j)$, and let us prove that $u> 0$ on $\Omega$, as well as \eqref{hopf}. The case  $J = \emptyset$ will follow too. Consider, for $\delta > 0$, the level set $N_\delta = \{f = \delta\}$. Since $f$ is smooth and $\abs{\nabla f} > 0$ on $N$, then for $\delta > 0$ small enough $N_\delta$ is a regular level set of $f$. Define, on $\Omega_\delta = \Omega \setminus \{f < \delta\}$, the function $w = {u}/{f}$. Then, it is directly checked that $w$ satisfies
\begin{equation}
\label{eq-furba}
\Delta w + 2 \langle \nabla w, \nabla f\rangle  =   - \frac{1}{f} < 0. 
\end{equation}
In particular, $w_\delta$ is subject to a strong minimum principle in $\Omega_\delta$. Since $u$ is continuous up to $N_j$, where it takes the positive value $c_j$, the function $w_\delta > 0$ on $N_\delta \cap \Omega$, and thus its minimum value on $\partial \Omega$ is zero. Thus, $u > 0$ on $\Omega_\delta$. Letting $\delta \to 0^+$, by the continuity up to the boundary of $u$ we deduce that $u > 0$ on $\Omega.$
Moreover, again because of \eqref{eq-furba} and by the Hopf Lemma for supersolutions to elliptic PDE's, we get 
\begin{equation}
\label{hopf-inproof}
0 > \frac{\partial w}{\partial \nu}(x) = \frac{1}{f} \left\langle \nabla u - \frac{1}{f} u \nabla f, \nu \right\rangle(x) =   \frac{1}{f} \frac{\partial u}{\partial \nu}(x),  
\end{equation}
for any $x \in \Sigma$, as $u = 0$ on such hypersurface. This concludes also the proof of the claimed Hopf-type property.
\end{proof}

\subsection{The vector field X and its nonnegative divergence}
On a Riemannian manifold $(M, g)$ consider for an open $U\subset M$, a \emph{positive} function $f \in C^\infty (U)$. Then, for a smooth function $u \in C^\infty(U)$ solving 
\begin{equation}
    \label{torsione-dasola}
\Delta u = - 1 + \frac{\Delta f}{f} u,
\end{equation}
we define the vector field $X$ by
\begin{equation}
\label{fieldX}
    X = f \nabla \abs{\nabla u}^2 + \frac{2}{n} f \nabla u - \nabla\nabla f \nabla u^2 -\frac{2}{n} u \nabla f - 2 u\nabla\nabla u \nabla f + 2 u^2 \frac{\nabla\nabla f}{f} \nabla f.
\end{equation}
Its divergence is computed in the following proposition.

\begin{proposition}
\label{divX-prop}
Let $(M, g)$ be a Riemannian manifold, and let $X$ be defined as above on some open set $U \subset M$ endowed with a positive smooth function $f$. Then, we have
\begin{equation}
\label{divX}    
    \dive X = 2 f \left\vert \nabla \nabla u - \frac{\Delta u}{n} g - u\left(\frac{\nabla \nabla f}{f} - \frac{\Delta f}{nf} g\right) \right\vert^2  + 2Q\left(\nabla u - \frac{u}{f} \nabla f, \nabla u - \frac{u}{f} \nabla f\right),
\end{equation}
where the tensor $Q$ is defined as
\begin{equation}
\label{qinprop}
Q = f\ric - \nabla \nabla f + \Delta f g.     
    \end{equation}
\end{proposition}
\begin{proof}
Let us compute separately the divergence of the six summands forming $X$. By applying the Bochner identity, and using the relation \eqref{torsione-dasola} we have 
\begin{equation}
  \label{div1}
  \begin{split}
    \dive (f \nabla \abs{\nabla u}^2) &= \langle \nabla f, \nabla \abs{\nabla u}^2\rangle + f \Delta \abs{\nabla u}^2 \\ 
    &\!\!\!\!\!\!\!\!\!\!\!\!\!\!\!\!\!= \langle \nabla f, \nabla \abs{\nabla u}^2\rangle + 2f \left[\abs{\nabla \nabla u}^2 + \langle \nabla (\Delta u), \nabla u \rangle + \ric(\nabla u,\nabla u)\right] \\
    &\!\!\!\!\!\!\!\!\!\!\!\!\!\!\!\!\!=  \langle \nabla f, \nabla \abs{\nabla u}^2\rangle + 2f \left[\abs{\nabla \nabla u}^2 + \frac{1}{f}\langle \nabla \Delta f, \nabla u \rangle - \frac{1}{f^2} \langle \Delta f, \nabla u \rangle + \frac{\Delta f}{f} \abs{\nabla u}^2 + \ric(\nabla u,\nabla u) \right]
\end{split}
\end{equation}
Observe now that, again since $u$ solves \eqref{torsione-dasola}, we have
\begin{equation}
    \label{quadratone-modificato}
    \begin{split}
    \left\vert \nabla \nabla u - \frac{\Delta u}{n} g - u\left(\frac{\nabla \nabla f}{f} - \frac{\Delta f}{nf} g\right) \right\vert^2 &= \left\vert \nabla \nabla u - u\frac{\nabla \nabla f}{f} + \frac{1}{n} g \right\vert^2 = \\ 
    &\!\!\!\!\!\!\!\!\!\!\!\!\!\!\!\!\!\!\!\!\!\!\!\!\!\!\!\!\!\!\!\!\!\!\!= \left\vert \nabla \nabla u \right\vert^2 + u^2 \frac{\left\vert\nabla \nabla f\right\vert^2}{f^2} + \frac{1}{n} - 2 u \left\langle\frac{\nabla \nabla f}{f}, \nabla \nabla u \right\rangle + \frac{2}{n} \Delta u - \frac{2}{n}\frac{\Delta f}{f} u \\
    &\!\!\!\!\!\!\!\!\!\!\!\!\!\!\!\!\!\!\!\!\!\!\!\!\!\!\!\!\!\!\!\!\!\!\!\!= \left\vert \nabla \nabla u \right\vert^2 + u^2 \frac{\left\vert\nabla \nabla f\right\vert^2}{f^2} - 2 u \left\langle\frac{\nabla \nabla f}{f}, \nabla \nabla u \right\rangle + \frac{1}{n} \Delta u - \frac{1}{n}\frac{\Delta f}{f} u.
\end{split}
\end{equation}
Thus, rewriting $\abs{\nabla \nabla u}^2$ through \eqref{quadratone-modificato}, and plugging it into \eqref{div1}, we get
\begin{equation}
    \label{div1-bis}
    \begin{split}
     \dive (f \nabla \abs{\nabla u}^2) = \langle \nabla f, \nabla \abs{\nabla u}^2\rangle &+ 2  \left\vert \nabla \nabla u - \frac{\Delta u}{n} g - u\left(\frac{\nabla \nabla f}{f} - \frac{\Delta f}{nf} g\right) \right\vert^2 -2 u^2 \frac{\vert\nabla \nabla f\vert^2}{f} \\
     &+4u \left\langle\nabla\nabla f, \nabla \nabla u \right\rangle - \frac{2}{n}f\Delta u + \frac{2}{n} u\Delta f  + 2\langle \nabla \Delta f, \nabla u \rangle \\
     &-2\frac{\Delta f}{f} \langle \nabla f, \nabla u \rangle + 2 {\Delta f} \abs{\nabla u}^2 + 2f \ric (\nabla u, \nabla u),
\end{split}
\end{equation}
that is the formula for the divergence of the first term in \eqref{fieldX} we are going to employ in this proof.
The divergence of $f \nabla u$ is just written as
\begin{equation}
    \label{div2}
    \dive{f \nabla u} = \langle \nabla f, \nabla u\rangle + f \Delta u,
\end{equation}
with no need for employing the equation solved by $u$. Let us recall now the classical identity for the Ricci tensor given by
\begin{equation}
    \label{ricci-id}
    \Delta \nabla v = \nabla \Delta v + \ric \nabla v
\end{equation}
in force for any $C^{3}$-function.
Applying it to the smooth function $f$, we get, in taking the divergence of the third term,
\begin{equation}
    \label{div3}
    \begin{split}
    \dive(\nabla \nabla f \nabla u^2) &= 2 u\langle\Delta \nabla f, \nabla u\rangle + 2 u \langle \nabla\nabla f, \nabla \nabla u\rangle + 2  \nabla \nabla f(\nabla u, \nabla u) \\ 
&=2 u\langle\Delta \nabla f, \nabla u\rangle + 2u \ric(\nabla f, \nabla u) + 2 u \langle \nabla\nabla f, \nabla \nabla u\rangle + 2  \nabla \nabla f(\nabla u, \nabla u).  
\end{split}
\end{equation}
The fourth divergence is  kept as
\begin{equation}
    \label{div4}
    \dive (u \nabla f) = \langle \nabla u, \nabla f\rangle + u \Delta f.
\end{equation}
In computing the divergence of the fifth term, we apply the identity \eqref{ricci-id} to $u$, and combine it with \eqref{torsione-dasola}, to get
\begin{equation}
    \label{div5}
    \begin{split}
    \dive (u \nabla \nabla u \nabla f) &= u \langle \Delta \nabla u, \nabla f\rangle + u \langle \nabla \nabla u, \nabla \nabla f\rangle + \nabla \nabla u (\nabla u, \nabla f) \\ 
    &\!\!\!\!\!\!= \frac{u}{f}\langle \nabla \Delta f, \nabla f\rangle - \frac{u}{f^2} \Delta f \abs{\nabla f}^2 + u\ric (\nabla u, \nabla f)  + u \langle \nabla \nabla u, \nabla \nabla f\rangle +  \frac{1}{2} \langle \nabla \abs{\nabla u}^2, \nabla f \rangle.
\end{split}
\end{equation}
We compute the last term, again using \eqref{ricci-id} in relation with $f$, as
\begin{equation}
    \label{div6}
    \begin{split}
    \dive (u^2 \frac{\nabla \nabla f}{f} \nabla f) = \frac{u^2}{f} \langle \nabla \Delta f, \nabla f\rangle + \frac{u^2}{f} \ric(\nabla f, \nabla f) &+ 2\frac{u}{f} \nabla \nabla f(\nabla u, \nabla f) \\
    &- \frac{u^2}{f^2} \nabla \nabla f(\nabla f, \nabla f) + \frac{u^2}{f}\abs{\nabla \nabla f}^2.
\end{split}
\end{equation}
The identity \eqref{divX} now follows immediately by putting together \eqref{div1-bis}, \eqref{div2}, \eqref{div3}, \eqref{div4}, \eqref{div5} and \eqref{div6}. 
\end{proof}
In the following Remark, we clarify the link between the nonnegative divergence of $X$ and the sub-harmonicity of the P-function \cite{qiu-xia-pfunz,Ciraolo-Vezzoni-Serrin} in space forms.
\begin{remark}[Relation between $X$ and the P-function of space forms.]
\label{XeP}
Let $\mathbb{M}^n_K$ be the hyperbolic space $\Heis^n_K$ or the hemisphere $\Sf^n_K$ of negative or positive constant sectional curvature $K$, respectively, and of dimension $n$. Let $\mathbb{M}^n_0$ be flat $\R^n$. These manifolds can be realized as warped products $(I_K \times \Sf^{n-1}, d r \otimes \! d r + h_K^2(r)g_{\Sf^{n-1}})$, where 
\begin{equation}
    I_K = 
    \begin{cases}
    [0, \infty) \qquad \text{if $K \leq 0$}\\
    \Big[0, \frac{\pi}{2\sqrt{K}}\Big) \qquad \text{if $K > 0$}
    \end{cases}
    \end{equation}
and
\begin{equation}
h_K=
\begin{cases}
r \qquad \text{if $K= 0$} \\
\frac{1}{\sqrt{\abs{K}}} \sinh(r\sqrt{\abs{K}}) \quad \text{if $K < 0$}\\
\frac{1}{\sqrt{K}} \sin(r \sqrt{K}) \quad \text{if $K > 0$}.
\end{cases}
\end{equation}
Letting $f = \dot h (r)$, these manifolds are immediately seen to satisfy $\nabla \nabla f = n^{-1}\Delta f g = -K f g$, to have constant Ricci curvature $\ric = (n-1)K$, and thus in particular to be static, that is $Q$ vanishes. Plugging this information into \eqref{divX}, a routine calculation yields 
\begin{equation}
\label{divX-spaceform}    
\dive (X) = f \dive \big(\nabla \abs{\nabla u}^2 + \frac{2}{n} \nabla u + K \nabla u^2\big) = 2f \left\vert \nabla \nabla u - \frac{\Delta u}{n} g\right\vert^2.
\end{equation}
On the other hand, we have 
\[
\nabla \abs{\nabla u}^2 + \frac{2}{n} \nabla u + K \nabla u^2 = \nabla P_K,
\]
where $P_K$ is the P-function for space forms
\begin{equation}
P_K = \abs{\nabla u}^2 + \frac{2}{n}u + Ku^2,
\end{equation}
and thus we deduced
\[
\Delta P_K = 2 \left\vert \nabla \nabla u - \frac{\Delta u}{n} g\right\vert^2,
\]
that is the sub-harmonicity stated in \cite[Lemma 2.1]{Ciraolo-Vezzoni-Serrin}.
\end{remark}

\section{Integral identities and geometric consequences}
\label{sec3}
We specialize now $X$ on a substatic Riemannian manifold with empty boundary or horizon boundary $N = \{f = 0\}$. 
The following straightforward application of the Divergence Theorem to $X$ in the subset $\Omega$ will naturally lead to the main results, Theorems \ref{alexandrov-th} and \ref{heintze-karcher}. Although the computation of the divergence of $X$ in Proposition \ref{divX-prop} is obviously carried out where $f \neq 0$, the identities that follow make sense continuously up to $N = \{f = 0\}$ since $f$ appears at the denominator only as $\nabla \nabla f / f$ and $\Delta f / f$, that by our Definition \ref{horizon} of horizon boundary  extend continuously at points where $f= 0$.
\begin{proposition}[Main integral identity]
\label{integral-id}
Let $(M, g)$ be a Riemannian manifold endowed with a nonnegative function $f \in C^{\infty}(M)$ possibly with horizon boundary $N = N_1 \sqcup \dots \sqcup N_l$. Let $\Omega \subset M$ be a bounded subset of $M$, with $C^{2, \alpha}$-boundary $\partial \Omega = \Sigma \sqcup \cup_{j \in J} N_j$ for $J \subseteq \{1, \dots, l\}$, with $\Sigma \cap N = \emptyset$, and let $u$ be a solution to \eqref{torsion-pb} with $c_j > 0$ for any $j \in J$. Then, we have
\begin{equation}
    \label{integral-f}
    \begin{split}
    \int_{\Omega} \dive X  \dd\mu \, + & \, 2\sum_{j \in J}\left[\frac{(n-1)}{n} \int_{N_j} u \abs{\nabla f} \dd\sigma - \int_{N_j} u^2 \abs{\nabla f} \left(\frac{\Delta f}{f} - \frac{\nabla \nabla f}{f} \left(\frac{\nabla f}{\abs{\nabla f}}, \frac{\nabla f}{\abs{\nabla f}}\right)\right) \dd\sigma \right]= \\ 
    &=-2 \int_{\Sigma} f \abs{\nabla u}^2 \HH \dd\sigma + 2\frac{(n-1)}{n} \int_{\Sigma} f \abs{\nabla u} \dd\sigma , 
\end{split}
\end{equation}
where $\HH$ is the mean curvature of $\Sigma$.
\end{proposition}
\begin{proof}
The Divergence Theorem applied to the vector field $X$ on $\Omega$ states that
\begin{equation}
    \label{div-applied-stupida}
    \int_\Omega \dive X \dd\mu = - \int_\Sigma \left\langle X, \frac{\nabla u}{\abs{\nabla u}} \right\rangle \dd\sigma -\sum_{j \in J} \int_{N_j} \left\langle X, \frac{\nabla f}{\abs{\nabla f}}\right\rangle \dd\sigma.
\end{equation}
Indeed, the vector field $\nabla u/ \abs{\nabla u}$ is normal to $\Sigma$, as it is a regular level set of $u$ by Theorem \ref{existence}, and points inside $\Omega$, and the same goes for ${\nabla f}/{\abs{\nabla f}}$ on $N_j$ by assumption.
We thus compute the two boundary integrals, starting from that on $\Sigma$. We have
\begin{equation}
\label{integralbordo1}
\begin{split}
\int_\Sigma \left\langle X, \frac{\nabla u}{\abs{\nabla u}} \right\rangle \dd\sigma &=  \int_\Sigma f \left \langle\nabla \abs{\nabla u}^2, \frac{\nabla u}{\abs{\nabla u}} \right\rangle \dd\sigma + \frac{2}{n} \int_\Sigma f \abs{\nabla u} \dd\sigma \\ 
&=   2\int_\Sigma f \abs{\nabla u} \left(\frac{\nabla u}{\abs{\nabla u}}, \frac{\nabla u}{\abs{\nabla u}}\right) \dd\sigma + \frac{2}{n} \int_\Sigma f   \abs{\nabla u} \dd\sigma.
\end{split}
\end{equation}
Recall now the classical relation between the Laplacian $\Delta_S$ induced on a hypersurface $S$ and that of the ambient metric
\begin{equation}
    \label{relation-laplace}
    \Delta_S \, v = \Delta v - \nabla \nabla v (\nu, \nu) - \HH \, \langle \nabla u, \nu \rangle \end{equation}
    for a $C^2$-function $v$, where $\nu$ is a unit normal to $S$, and $\HH$ the relative mean curvature. Applying it to $u$ on $\Sigma$, using that $u = 0$ on $\Sigma$ and that it solves $\Delta u = - 1 + u\Delta f/ f $, we deduce 
    \begin{equation}
    \label{mean-curv-u}    
        \nabla \nabla u \left(\frac{\nabla u}{\abs{\nabla u}}, \frac{\nabla u}{\abs{\nabla u}} \right) = - 1 + \HH \abs{\nabla u}
    \end{equation}
that, plugged into \eqref{integralbordo1}, yields
\begin{equation}
    \label{integralbodo1finale}
    \int_\Sigma \left\langle X,  \frac{\nabla u}{\abs{\nabla u}}\right\rangle \dd\sigma = 2 \int_\Sigma f \abs{\nabla u}^2 \, \HH \dd\sigma - 2 \,\frac{(n-1)}{n} \int_\Sigma f \abs{\nabla u} \dd\sigma.
\end{equation}
We now turn our attention to the integrals on $N_j$ for $j \in J$, in the right hand side of \eqref{div-applied-stupida}. Using that $N_j \subset \{f = 0\}$ is regular for $f$, and thus its normal $\nu_j$ pointing inside $\Omega$ is given by $\nu_j(x) = \nabla f /\abs{\nabla f}(x)$ for any $x \in N_j$, we get
\begin{equation}
    \label{integralbordo2}
    \begin{split}
 \int_{N_j} \left\langle X, \frac{\nabla f}{\abs{\nabla f}}\right\rangle \dd\sigma = - \frac{2}{n} \int_{N_j}  u \abs{\nabla f}  \dd\sigma &- 2\int_{N_j} u \abs{\nabla f} \, \nabla\nabla u \left(\frac{\nabla f}{\abs{\nabla f}} , \frac{\nabla f}{\abs{\nabla f}}\right)  \dd\sigma  \\
 &+ 2 \int_{N_j} u^2 \abs{\nabla f} \frac{\nabla \nabla f}{f}\left(\frac{\nabla f}{\abs{\nabla f}}, \frac{\nabla f}{\abs{\nabla f}} \right) \dd\sigma. 
\end{split}
\end{equation}
Using again \eqref{relation-laplace}, this time applied to the function $u$ on $N_j$, where $u$ is still constant, we deduce
\begin{equation}
    \label{mean-curv-f}
    \nabla\nabla u \left(\frac{\nabla f}{\abs{\nabla f}} , \frac{\nabla f}{\abs{\nabla f}}\right) = - 1 +  \frac{\Delta f}{f} u,
    \end{equation}
    where we exploited the minimality of $N_j$. Plugging it into \eqref{integralbordo2}, we get
    \begin{equation}
        \label{integralbordo2finale}
      \int_{N_j} \left\langle X, \nu_j\right\rangle \dd\sigma = 2 \frac{(n-1)}{n} \int_{N_j} u \abs{\nabla f} \dd\sigma - \int_{N_j} u^2 \abs{\nabla f} \left(\frac{\Delta f}{f} - \frac{\nabla \nabla f}{f} \left(\frac{\nabla f}{\abs{\nabla f}}, \frac{\nabla f}{\abs{\nabla f}}\right)\right) \dd\sigma,    
    \end{equation}
    that, combined with \eqref{integralbordo1} into \eqref{div-applied-stupida}, provides \eqref{integral-f}.
\end{proof}
The above result should be compared with \cite[Theorem 2.1]{magnanini-poggesi1}, that yields \eqref{integral-f} in the very special geometry of $\R^n$.

As a first application of Proposition \ref{integral-id}, we observe how, in the substatic case, in the mere presence of a \emph{strictly mean-convex} hypersurface $\Sigma$ homologous to $\cup_{j \in J} N_jFs$, a purely geometric integral quantity computed on the horizon $N$ is positive. This is what will enable us to choose the sharpest constants $c_j$ in \eqref{torsion-pb}. With $S$ being homologous to $\cup_{j \in J} N_j$, we mean that there exists a a bounded open set $E$ with $\partial E = S \sqcup (\cup_{j \in J} N_j)$.
\begin{proposition}
\label{boundary-positive}
Let $(M, g)$ be a substatic Riemannian manifold with substatic potential $f$ and nonempty horizon boundary $N= N_1 \sqcup \dots \sqcup N_l$. Assume there exists a $C^{2, \alpha}$ strictly mean-convex hypersurface $S$ homologous to $\cup_{j \in J} N_j$. Then
\begin{equation}
\label{boundary-positive-integral}    
\int_{N_j} \abs{\nabla f}\left[\frac{\Delta f}{f} - \frac{\nabla \nabla f}{f} \left(\frac{\nabla f}{\abs{\nabla f}}, \frac{\nabla f}{\abs{\nabla f}}\right)\right] \dd\sigma > 0
\end{equation}
for any $j \in J$.
\end{proposition}
\begin{proof}
Assume by contradiction that there exists $i \in J$ such that \eqref{boundary-positive-integral} fails. Then, let $E$ be the subset bounded by the disjoint $\Sigma$ and $\cup_{j \in J} N_j$, and let $u$ be a solution to \eqref{torsion-pb} for some positive $c_1, \dots, c_l$. Apply \eqref{integral-f} in $E$ to get
\begin{equation}
\label{inequality-positivity-c}
\begin{split}
2\frac{(n-1)}{n}\int_{N_i} c_i \abs{\nabla f} \dd\sigma \leq \sum_{j \in (J \setminus \{i\})} & c_j\int_{N_j} \abs{\nabla f}\left[\frac{\Delta f}{f} - \frac{\nabla \nabla f}{f} \left(\frac{\nabla f}{\abs{\nabla f}}, \frac{\nabla f}{\abs{\nabla f}}\right)\right] \dd\sigma  \\
&- 2 \inf_S f \inf_S\HH \int_{S}  \abs{\nabla u}^2 \dd\sigma  + 2\frac{(n-1)}{n} \sup_S f \int_S \abs{\nabla u} \dd\sigma,
\end{split}
\end{equation}
where we exploited the substaticity of $(M, g)$ to infer the nonnegativity of the divergence of $X$ computed in Proposition \ref{divX-prop}, the assumption on the mean curvature of $S$ and the strict positivity of $f$ outside $N$. We are going to derive a contradiction by letting $c_i \to +\infty$. Let then $\{c_{ik}\}_{k \in \N}$ be a diverging sequence of positive numbers, and let $u_k$ be the corresponding solutions of \eqref{torsion-pb} (the other constants $c_j$ with $j \neq i$ remain unchanged). Observe that the left hand side of \eqref{inequality-positivity-c} diverges as $k \to + \infty$, and thus the above inequality immediately yields a contradiction if the integral of $\abs{\nabla u_k}$ remains bounded as $k \to + \infty$. We are thus left to consider the case where the integral of $\abs{\nabla u_k}$ diverges as $k \to + \infty$. By the H\"older inequality, we can obviously deduce that the same happens for the integral of $\abs{\nabla u_k}^2$, and estimating the right hand side of \eqref{inequality-positivity-c} with that we get
\begin{equation}
  c_{ik} \int_{N_i}  \abs{\nabla f} \dd\sigma \leq K_1 \left(\int_{S} \abs{\nabla u_k}^2 \dd\sigma\right)^{1/2} - K_2 \int_{S} \abs{\nabla u_k}^2 \dd\sigma,
\end{equation}
where we absorbed the quantities in front of the integrals, depending only on $S$, in the positive constants $K_1$ and $K_2$.
The right hand side diverging to $- \infty$ as $k \to + \infty$, against the left hand side diverging at $ + \infty$ yields the desired contradiction.
\end{proof}
\begin{remark}[Specialization to static metrics]
\label{remark-static} 
On a static metric with nonempty horizon boundary, that is with the tensor $Q$ constantly vanishing and with $N \neq \emptyset$, the geometric inequality \eqref{boundary-positive-integral} simplifies as follows. The metric being static, and $f$ vanishing on $N$, we have $\nabla \nabla f = \Delta f g$ on $N$, where thus the Hessian of $f$ vanishes. In particular, this implies that $N$ is totally geodesic and hence minimal, and this why the request of minimality for $N$ can be dropped when the metric is static. Letting $Y$ be a vector field that is tangent to $N$, we get $\langle \nabla \abs{\nabla f}^2, Y\rangle = 0$, and so we deduce that $\abs{\nabla f}$ is constant on $N$. This positive number is usually referred to as \emph{surface gravity} in the General Relativity literature. Moreover, again because of the vanishing of $Q$, the term in square brackets in \eqref{boundary-positive-integral} coincides with $- \ric (\nu, \nu)$, where $\nu = \nabla f / \abs{\nabla f}$. Consequently, in the special case of static metrics, we proved that if there exists a strictly mean-convex hypersurface homologous to $N$ then
\[
\int_N \ric (\nu, \nu) \dd\sigma < 0,
\]
as observed in \cite[Theorem 1.6]{wang-brendle} without an explicit proof.
\end{remark}
We do not know whether the minimality of $N$ and the substaticity of $(M, g)$ somehow suffice to produce a strictly mean-convex hypersurface as required above. However, this is better understood under asymptotic conditions at infinity. Going beyond the asymptotic flatness or asymptotic hyperbolicity, that naturally yield in the asymptotic region a foliation of strictly mean-convex hypersurfaces, we observe that, as a consequence of the recent work on the least area problem with obstacle \cite{Fog_Maz} and of a nice approximation lemma via Mean Curvature Flow \cite{Hui_Ilm}, in a noncompact Riemannian manifold of dimension $2 \leq n \leq 7$ and with \emph{outermost boundary}, the existence of a positive isoperimetric constant suffices to this aim. Let us recall that the boundary of a Riemannian manifold $M$ is called \emph{outermost} if there are no closed minimal hypersurfaces compactly contained in $M \setminus N$.  Such result holds actually regardless of curvature conditions on the Riemannian manifold with boundary.  \begin{proposition}
\label{existence-convex}
Let $(M, g)$ be a noncompact Riemannian manifold of dimension $2 \leq n \leq 7$ with smooth, compact and outermost boundary $N$. Assume there exists a positive isoperimetric constant $C_{\mathrm{iso}}$, that is
\begin{equation}
\label{iso-cond}
\frac{{\abs{\partial E}}^{n}}{\abs{E}^{n-1}} \geq C_{\mathrm{iso}}  
\end{equation}
for any bounded $E \subset M$ with smooth boundary, where we are denoting with $\abs{\partial E}$ and $\abs{E}$ the $(n-1)$-dimensional and the $n$-dimensional Hausdorff measure of $\partial E$ and $E$ respectively. Then there exists a strictly mean-convex smooth hypersurface $\Sigma$ homologous to $N$. 
\end{proposition}
\begin{proof}
Let $d: M \to [0, + \infty)$ be the distance function from the boundary $N$. By the smoothness of such hypersurface, there exists $\delta > 0$ small enough such that $N_\delta = \{f = \delta\}$ is a smooth $(n-1)$-dimensional hypersurface. Indeed, by \cite[Proposition 5.17]{mantegazza-distance}, $d$ is smooth in neighbourhood of $N$, and thus it classically solves its defining equation $\abs{\nabla d} = 1$, and the regularity of $\{f = \delta\}$ follows. By \cite[Theorem 1.1]{Fog_Maz}, we can solve the least area problem among sets containing $E_\delta =\{f \leq \delta\}$, and get a set $E_\delta^*$ with boundary $N \sqcup \Sigma_\delta$ , where $\Sigma_\delta$, due to our dimensional restriction, enjoys $C^{1, 1}$-regularity by \cite{sternberg-williams} (see also the statement \cite[Theorem 2.18]{Fog_Maz}). Observe that, although all the work in \cite{Fog_Maz} has been carried out in complete manifolds without boundary, it extends without any further effort to the present case. Since $\Sigma_\delta$ minimizes the area among outward variations, we immediately get by considering its first variation that its weak mean curvature (classically defined almost everywhere due to the $C^{1,1}$-regularity) is nonnegative. Observe that $\Sigma_\delta$ cannot be minimal, because of the boundary $N$ being outermost. Then, we can approximate $\Sigma_\delta$ in $C^1$ by smooth strictly mean convex hypersurfaces, as done in \cite[Lemma 5.6]{Hui_Ilm}, and let $\Sigma$ be one of these approximators.     
\end{proof}
We come back now to the proof of the main result. 
\begin{theorem}[Integral identity for Alexandrov's Theorem]
\label{alexandrov-th}
Let $(M, g)$ be a substatic Riemannian manifold with substatic potential $f$ and with horizon boundary $N = N_1 \sqcup  \dots \sqcup  N_l$. Assume there exists a $C^{2, \alpha}$ strictly mean-convex hypersurface $S$ homologous to $\cup_{j \in J} N_j$ . Let $\Omega \subset M$ be any bounded subset with $C^{2, \alpha}$-boundary $\partial \Omega = \Sigma \sqcup (\cup_{j \in J} N_j)$, for $J \subseteq \{1, \dots, l\}$, with $\Sigma \cap N = \emptyset$, and let $u$ be the solution to \eqref{torsion-pb} with
\begin{equation}
    \label{def-cj}
    c_j = \frac{\int_{N_j} \abs{\nabla f} \dd\sigma}{\int_{N_j}\abs{\nabla f}\left[\frac{\Delta f}{f} - \frac{\nabla \nabla f}{f} \left(\frac{\nabla f}{\abs{\nabla f}}, \frac{\nabla f}{\abs{\nabla f}}\right)\right] \dd\sigma}
    \end{equation}
for any $j \in J$. Then, we have
\begin{equation}
    \label{main-id}
    \begin{split}
     0 &\leq \int_\Omega  f \left\vert \nabla \nabla u - \frac{\Delta u}{n} g - u\left(\frac{\nabla \nabla f}{f} - \frac{\Delta f}{nf} g\right) \right\vert^2  + Q\left(\nabla u - \frac{u}{f} \nabla f, \nabla u - \frac{u}{f} \nabla f\right) \dd\mu \\ 
     & \qquad\qquad\qquad\qquad +  \frac{(n-1)}{n} \frac{1}{R} \int_\Sigma f(R - \abs{\nabla u})^2 \dd\sigma 
     = \int_\Sigma f \abs{\nabla u}^2 (\overline{\HH} - \HH) \dd\sigma,
    \end{split}
    \end{equation}
    where
   $Q = f\ric - \nabla\nabla f +\Delta f g \geq 0$
    and
    \begin{equation}
        \label{ReH}
        R= \frac{\int_\Omega f\dd\mu + \sum_{j \in J}c_j\int_{N_j} \abs{\nabla f}\dd\sigma}{\int_{\Sigma} f \dd\sigma} \qquad \overline{\HH} = \frac{(n-1)}{n} \frac{1}{R}.
    \end{equation}
    In particular, if $\HH \geq \overline{\HH}$ on $\Sigma$, then $\Sigma$ is a totally umbilical hypersurface.
\end{theorem}
\begin{proof}
First observe that, under the above assumptions, $c_j$ is well defined due to Proposition \ref{boundary-positive}. Then, plugging $u$ with such boundary data into \eqref{integral-f}, we see that the integrals on $N_j$ cancel out. Moreover, by Proposition \ref{divX-prop}, the integral on $\Omega$ in the central term of the chain \eqref{main-id} coincides with the integral of the divergence of $X$. 

Let us now compute the boundary integral in such central term. 
Observe that, on the one hand, we have by Divergence Theorem, the regularity of $\Sigma$ as $0$-level set of $u$ and $N$ being a horizon,
\begin{equation}
    \label{relation-boundaries}
    \int_\Omega \dive (f \nabla u) \dd\mu = - \int_\Sigma f\abs{\nabla u}\dd\sigma,
\end{equation}
while on the other hand we can compute the same term as
\begin{equation}
    \label{relation-boudary2}
    \begin{split}
    \int_\Omega \dive(f \nabla u) \dd\mu &= \int_\Omega \langle \nabla f, \nabla u\rangle \dd\mu - \int_\Omega f \dd\mu + \int_\Omega \Delta f u \dd\mu \\
    &= - \int_\Omega f \dd\mu - c_j\sum_{j \in J}\int_{N_j}  \abs{\nabla f} \dd\sigma,
\end{split}
\end{equation}
where in the first identity we employed the differential relation in \eqref{torsion-pb} satisfied by $u$, while the second one is an integration by parts. Comparing \eqref{relation-boundaries} with \eqref{relation-boudary2}, we get
\begin{equation}
    \label{relation-boundaryfinal}
    \int_\Omega f \dd\mu = \int_{\Sigma} f \abs{\nabla u} \dd\sigma - \sum_{j \in J}c_j\int_{N_j}  \abs{\nabla f} \dd\sigma.
\end{equation}
The above identity immediately implies that, with the choice of $R$ done in \eqref{ReH}, we have
\begin{equation}
    \label{last-term}
    \frac{1}{R} \int_\Sigma f(R - \abs{\nabla u})^2 \dd\sigma  = \frac{1}{R} \int_{\Sigma} f \abs{\nabla u}^2 \dd\sigma - \int_\Sigma f \abs{\nabla u} \dd\sigma.
\end{equation}
The derivation of \eqref{main-id} from \eqref{integral-f} is now straightforward.

\medskip

For what it concerns the total umbilicality of $\Sigma$ in case $\HH \geq \overline{\HH}$, just observe that in this case we immediately deduce that
\[
\left\vert \nabla \nabla u - \frac{\Delta u}{n} g - u\left(\frac{\nabla \nabla f}{f} - \frac{\Delta f}{nf} g\right) \right\vert^2 = 0
\]
on $\Omega$. In particular, since $u \in C^{2, \alpha}$ up to $\Sigma$, where $u = 0$, we deduce that 
\[
\nabla \nabla u - \frac{\Delta u}{n} g = 0
\]
on $\Sigma$, that we recall being a regular level set for $u$. Thus, we have
\begin{equation}
\hh_{ij} - \frac{\HH}{n-1} g_{ij} = \frac{1}{\abs{\nabla u}} \left(\nabla \nabla u - \frac{\Delta u}{n} g\right)(e_i, e_j) = 0,   
\end{equation}
for any $i, j \in\{1, \dots, n-1\}$, where $\hh$ is the second fundamental form of $\Sigma$ and $\{e_i\}_{i \in \{1, \dots, n-1\}}$ form an  orthonormal basis for the tangent space of $\Sigma$ at any of its points. This is the umbilicality of $\Sigma$. 
\end{proof}
The Heintze-Karcher inequality  with explicit $L^2$-deficit follows from Propositions \ref{main-id} and \ref{boundary-positive} as well.
\begin{theorem}[Heintze-Karcher inequality with $L^2$-deficit]
\label{heintze-karcher}
Let $(M, g)$ be a substatic Riemannian manifold with substatic potential $f$ and with horizon boundary $N = N_1 \sqcup  \dots \sqcup  N_l$. Let $\Omega \subset M$ be any bounded subset with $C^{2, \alpha}$-boundary $\partial \Omega = \Sigma \sqcup (\cup_{j \in J} N_j)$, for $J \subseteq \{1, \dots, l\}$, with $\Sigma \cap N = \emptyset$, such that $\Sigma$ is \emph{strictly mean-convex}. Let $u$ be the solution to \eqref{torsion-pb} with $c_j$ defined as in \eqref{def-cj} for $j \in J$. Then, we have
\begin{equation}
\label{heintze-karcherf}
\begin{split}
0 \, \leq \, \,  \frac{n}{n-1}  & \int_\Omega   \bigg|\nabla \nabla u - \frac{\Delta u}{n} g - u\left(\frac{\nabla \nabla f}{f}  - \frac{\Delta f}{nf} g\right)\bigg\vert^2 + Q\left(\nabla u - \frac{u}{f} \nabla f, \nabla u - \frac{u}{f} \nabla f\right) \dd\mu \\
&+\frac{n-1}{n}\int_{\Sigma}\frac{1}{\HH}\left(1-\frac{n}{n-1}\HH \abs{\nabla u}\right)^2 \dd\sigma = \frac{n-1}{n} \int_{\Sigma}\frac{f}{\HH} \dd\sigma -\int_{\Omega} f \dd\mu \\ 
& \qquad\qquad\qquad\qquad\qquad\qquad\qquad\qquad\qquad\qquad\qquad\qquad\! - \sum_{j \in J} c_j \int_{N_j} \abs{\nabla f} \dd\sigma.
\end{split}
\end{equation}
In particular, the Heintze-Karcher inequality
\begin{equation}
    \label{heintze-vera}
     \frac{n-1}{n} \int_{\Sigma}\frac{f}{\HH} \dd\sigma \geq\int_{\Omega} f \dd\mu +  \sum_{j \in J} c_j \int_{N_j} \abs{\nabla f} \dd\sigma
\end{equation}
holds true, and equality is achieved only if $\Sigma$ is totally umbilical.
\end{theorem}
\begin{proof}
Observe first that $\Sigma$ itself furnishes the strictly mean-convex hypersurface fulfilling the assumption of Proposition \ref{boundary-positive}. Consequently, the constants $c_j$ in \eqref{def-cj} are well defined. The identity \eqref{heintze-karcherf} then follows from \eqref{integral-f} with such a choice of $c_j$, and by plugging in the identity
\begin{equation}
\left(\HH \abs{\nabla u}^2-\frac{n-1}{n}\abs{\nabla u}\right) = \frac{1}{\HH}\left(\frac{n-1}{n}\right)^2\left(1 -\frac{n}{n-1}\HH \abs{\nabla u}\right)^2 - \frac{1}{\HH}\left(\frac{n-1}{n}\right)^2 + \frac{(n-1)}{n} \abs{\nabla u}
\end{equation}
together with \eqref{relation-boundaryfinal}. 

\smallskip

The rigidity statement concerning the equality case in \eqref{heintze-vera} follows exactly as that in Theorem \ref{alexandrov-th}.
\end{proof}
Let us explicitly observe that the rigidity statement in Theorem \ref{alexandrov-th} can in fact be deduced from the identity case in the above Heintze-Karcher inequality.

\subsection{Improved rigidity statements}
\label{subsec}
Here, we draw some additional information on totally umbilical hypersurfaces in substatic manifolds with horizon boundary under some extra assumptions.

\smallskip

The following should be read as a generalization of conditions (H4) and (H4') in \cite{Brendle-alexandrov} allowing  the author to infer that, in substatic warped products with horizon boundary, constantly mean-curved hypersurfaces are actually cross-sections. In fact, the argument is an adaptation of the one give by in \cite[p. 265]{Brendle-alexandrov}. We refer to the Appendix for the explicit derivation of Brendle's results from ours.
\begin{proposition}
\label{rigidity-eigenvalue}
Let $(M, g)$ be a substatic Riemannian manifold with \emph{nonconstant} substatic potential $f$ and with horizon boundary $N = N_1 \sqcup \dots \sqcup N_l$ with $l \in \N$.  
Assume moreover that $\nabla f$ never vanishes on $M$ and that is an eigenvector of $\ric$ and that its eigenvalue $\lambda_f \neq \lambda$ any time $\ric_{\alpha\beta} Y_\beta = \lambda Y_\alpha$ for some other eigenvector $Y$ not parallel to $\nabla f$.
Then, if, for some $J \subseteq \{1, \dots, l\}$ we have that $\Sigma$ is a $C^{2, \alpha}$ hypersurface homologous to $\cup_{j \in J} N_j$ with constant mean curvature $\HH = \overline{\HH}$, where $\overline{\HH}$ is defined in \eqref{ReH}, then each connected component of $\Sigma$ is a connected component of a level set of $f$.
\end{proposition}
\begin{proof}
Observe first that by the rigidity statement in Theorem \ref{alexandrov-th} $\Sigma$ is an umbilical hypersurface. Then, denoting with $\nu = - \nabla u / \abs{\nabla u}$ the unit outer normal to $\Sigma$,
by the (traced) Codazzi-Mainardi equations we have on $\Sigma$ 
\begin{equation}
\label{codazzi}
\ric_{j\nu} = \nabla_i \hh_{ij} - \nabla_j \HH = \frac{1}{n-1}\nabla_j \HH - \nabla_j \HH = -\frac{n-2}{n-1}  \nabla_j \HH = 0,
\end{equation}
for any $j \in \{1, \dots, n-1\}$.
The second equality is due to the umbilicality of $\Sigma$, and the last one to its constant mean curvature. This implies that $\nu$ is actually an eigenvector of $\ric$. By our assumption on $\nabla f$, we deduce that, if $\nu$ were not parallel to $\nabla f$ on $\Sigma$, the eigenvalues of these two eigenvectors would differ, and thus they would be orthogonal. This possibility is however ruled out by the fact that since $\nabla f$ never vanishes, the level sets of $f$ form a foliation of $M$, and thus one can find a point on each connected component of $\Sigma$ where this hypersurface touches one of these level sets. On such point, in particular $\nu$ is parallel to $\nabla f$, and we deduce that these two vectors are parallel on the whole connected component of $\Sigma$. This implies that it coincides with a connected component of a level set of $f$.
\end{proof}

Let us now provide improved rigidity statements for hypersurfaces satisfying the special cases in Theorems \ref{alexandrov-th} and \ref{heintze-karcher} when the ambient manifold satisfies the strong requirement $\nabla \nabla f = (\Delta f g)/n$, generalizing in this way the full Alexandrov Theorem in space forms \cite{alexandrov-originale2} and the characterization of the equality case in the Heintze-Karcher inequality in space forms \cite{qiu-xia} and in manifolds with nonnegative Ricci curvature \cite[Theorem 1]{ros}. We finally observe that establishing a characterization for the equality case in \eqref{heintze-karcherf} directly suffices also to describe the case $\HH \geq \overline {\HH}$ in Theorem \ref{alexandrov-th}, since the latter directly implies that \eqref{heintze-karcherf}  holds true with equality sign. 
\begin{proposition}[Improved rigidity under a conformal Hessian condition]
\label{conf-hess}
Let $(M, g)$ be a substatic Riemannian manifold with substatic potential $f$ and with horizon boundary $N = N_1 \sqcup \dots \sqcup N_l$, with $l \in \N$. Assume that $f$ is noncostant and satisfies
$\nabla \nabla f = (\Delta f g)/n$ in $\Omega$, a bounded set with $C^{2, \alpha}$-boundary such that $\partial \Omega = \Sigma \sqcup (\cup_{j \in J} N_j)$, for $J \subseteq \{1, \dots, l\}$, with $\Sigma \cap N = \emptyset$. Then, if $\Omega$ satisfies the identity in the Heintze-Karcher inequality \eqref{heintze-vera}, we have that $\Sigma$ is either a geodesic sphere or it is a level set of the distance function from $\cup_{j \in J} N_j$. 

On the other hand, if $f$ is constant (and thus $(M, g)$ has nonnegative Ricci curvature) and $\Omega$ satisfies the identity in \eqref{heintze-vera}, then $\Omega$ is isometric to a ball in flat $\R^n$.
\begin{proof}
By means of \eqref{heintze-karcherf}, we have that under the above assumptions the Hessian of $u$ is proportional to $g$.  Thus, $g$ is isometric to a warped product, because of a classical result that is reported e.g. in \cite[Theorem 4.3.3]{Petersen_book} (see otherwise \cite{cat-man-maz} or \cite[Section 1]{Che-Cold}), at least in a neighbourhood of the regular level set $\Sigma$. Moreover, the natural construction of such isometry shows that $u$ and $\abs{\nabla u}$ depends only on the (signed) distance from $\Sigma$, and that the warping function vanishes when $\abs{\nabla u}$ does. In particular, if $\abs{\nabla u}$ vanishes when approaching some level set $\{u = a\}$, then such level set actually consists of a point $p$ where the metric extends smoothly. Hence, in this case, $\Sigma$ consists of points that are equidistant from $p$, and in other words it is a geodesic sphere. If otherwise the gradient of $u$ never vanishes in $\overline{\Omega}$, then the same reasoning allows to conclude that $\Sigma$ consists of points that are equidistant from $\cup_{j \in J} N_j$. 

\smallskip

If $f$ is constant, then in particular $N$ is empty, and the arguments used above imply that $(\overline{\Omega}, g)$ is isometric to a warped product metric $([0, a) \times S, d\rho + d\rho + h^2(\rho) g_S)$, where $(S, g_S)$ is a closed Riemannian manifold of dimension $n-1$. By the smoothness at $0$, we infer from the classical necessary and sufficient condition \cite[Proposition 1.4.7]{Petersen_book} that $(S, g_S)$ is isometric to the unitary sphere, and that $h(\rho)$ behaves as $\rho$ when $\rho \to 0^+$. Finally, employing $\ric(\nabla u, \nabla u) = 0$ again arising from \eqref{heintze-karcherf}, and recalling that by the construction above $\nabla u$ is orthogonal to the level sets of $\rho$, we get by the expression of the Ricci curvature in warped products (see e.g. \cite[Proposition 9.106]{Besse_book})  that $h(\rho)$ is linear in $\rho$ (see e.g. the explicit derivation carried out in the proof of \cite[Lemma 3.1]{Ago_Fog_Maz_1}). We conclude that $h(\rho) = \rho$, inferring this way the claimed isometry with a ball in flat $\R^n$.
\end{proof}
\end{proposition}
Actually the basic Minkowski identity in space forms implies that, in this setting, connected embedded hypersurfaces of constant mean curvature necessarily satisfies $\HH = \overline{\HH}$, and in particular the Heintze-Karcher inequality is saturated on every connected component of $\Omega$. This is why the above proposition completes the Alexandrov's characterization of embedded hypersurfaces in space forms. Such an application of the Minkowski identity will be described in more details in the Appendix, where we will take into account its generalization to warped products due to Brendle.

\setcounter{equation}{0}
\renewcommand\theequation{A.\arabic{equation}}
\section*{Appendix: Brendle's Alexandrov Theorem for warped products}
We briefly discuss how the Heintze-Karcher inequality and Alexandrov-type Theorem in \cite{Brendle-alexandrov} follow from the results contained in the previous sections. Actually, in this paper the author establishes a Heintze-Karcher inequality \cite[Theorem 3.5 and Theorem 3.1]{Brendle-alexandrov} for hypersurfaces in warped products satisfying certain geometric conditions, and derives the umbilicality of constantly mean-curved hypersurfaces \cite[Theorem 1.1 and Theorem 1.4]{Brendle-alexandrov} sitting in these spaces. Namely, for the ambient warped products spaces of the form $(M, g) = ([0, a) \times N, d\rho \otimes d\rho + h^2(\rho) g_N)$, of dimension $n \geq 3$,  where $a$ is a positive real number or $+ \infty$ and $h \in C^{\infty}([0, a))$ is nonnegative, the author considers the following conditions:
\begin{itemize}
    \item[(H0)] the cross-section $N$ is a compact manifold of dimension $n-1$ with $\ric_N \geq c (n-2) g_{N}$, where $c$ is some real constant,
    \item[(H1)] $h'(0) = 0$ and $h''(0) > 0$,
    \item[(H2)] $h'(\rho) > 0$ for any $\rho \in (0, a)$,
    \item[(H3)] the function
    \[
    2\frac{h''(\rho)}{h(\rho)} - (n-2)\frac{c - h'(\rho)^2}{h(\rho)^2}
    \]
    is nondecreasing for any $\rho \in (0, a)$,
    \item[(H4)] the function
    \[
    \frac{h''(\rho)}{h(\rho)} + \frac{c - h'(\rho)^2}{h(\rho)^2} \neq 0
    \]
    for any $\rho \in (0, a)$.
    \end{itemize}
    In case the cross section $N$ is given by $\Sf^{n-1}$, and $c = 1$, the  condition
    \begin{itemize}
        \item[(H1)'] $h$ satisfies
        \begin{equation}
            \label{condion-warped-smooth}
            h(\rho) = \rho \, \phi(\rho^2),
        \end{equation}
        for some  positive function $\phi \in  C^{\infty}[0, \sqrt{a})$ with $\phi(0) = 1$, is also taken into account in place of (H1')
    \end{itemize}
    
    \smallskip
    
A direct computation performed in the proof of \cite[Proposition 2.1]{Brendle-alexandrov} actually shows that a warped product manifold satisfying the conditions (H0), (H2) and (H3) is substatic with substatic potential $f$ defined by $f = h'(\rho)$, in the sense of Definition \ref{substatic-def}. Moreover, the condition (H1) coupled with (H2) implies that $\{\rho = \rho_0\}$ is a nonempty horizon boundary, in the sense of Definition \ref{horizon}, (i). In case $N = \Sf^{n-1}$ and (H1') holds, 
it is asserted in \cite[p. 268]{Brendle-alexandrov} that $g$ closes smoothly at $\rho = 0$. 
Hence, the combination of (H1') again with (H2) implies that we are in the situation of Definition \ref{horizon}, (ii).

\smallskip

Resuming, we infer that warped product manifolds satisfying (H0), (H1), (H2) and (H3), with (H1') possibly in place of (H1) when $N= \Sf^{n-1}$ and $c = 1$ are substatic with horizon boundaries, and thus our main results directly apply in this setting. In particular, \eqref{heintze-vera} recovers Brendle's Heintze-Karcher inequalities in \cite[Theorem 3.5 and Theorem 3.11]{Brendle-alexandrov} as well as their rigidity statements.

\medskip

For what it concerns the finer characterization of hypersurfaces with constant mean curvature as cross-sections of the warped products, let us first point out that, in the special warped product geometry, a version of the Minkowski identity holds true \cite[Proposition 2.3]{Brendle-alexandrov}. As a consequence of this, through a straightforward integration by parts (see e.g. \cite[p. 265]{Brendle-alexandrov}), hypersurfaces with constant mean curvature must satisfy $\HH = \overline{\HH}$, where, in our more general notation, $\overline{\HH}$ is given by \eqref{ReH}. Moreover, when the condition (H4) is added to (H0), (H1), (H2) and (H3), with (H1) possibly changed with (H1') when $N = \Sf^{n-1}$ and $c = 1$, then the explicit form of the Ricci tensor in warped products \cite[(2)]{Brendle-alexandrov} implies that the assumptions in Proposition \ref{rigidity-eigenvalue} are in force. Such result, applied with $f = h'(\rho)$, in particular implies that a connected hypersurface $\Sigma$ is homotothetical to the cross-section $N$, when the latter is connected too.

\smallskip

We also remark that in the space form case condition (H4) does not hold true, and thus a different argument like that reported in the proof of Proposition \ref{conf-hess} is needed in order to reach for such an improved rigidity statement.

\printbibliography

\end{document}